\documentclass[11pt,oneside,english]{amsart}

\usepackage[T1]{fontenc}
\usepackage[latin9]{inputenc}
\usepackage{amsthm}
\usepackage{amssymb}
\usepackage{graphicx}
\usepackage{geometry}
\usepackage{hyperref}
\usepackage{bbm}
\usepackage{enumitem}
\usepackage{verbatim}
\usepackage{babel}
\usepackage{textgreek}

\setlist[enumerate,1]{leftmargin=0.25in}
\setlist[itemize,1]{leftmargin=0.255in}


\theoremstyle{plain}
\newtheorem{theorem}{Theorem}

\theoremstyle{plain}
\newtheorem{lemma}[theorem]{Lemma}
  
\theoremstyle{plain}
\newtheorem{proposition}[theorem]{Proposition}
  
\theoremstyle{plain}
\newtheorem{corollary}[theorem]{Corollary}
  
\theoremstyle{plain}

\newtheorem*{proposition*}{Proposition}

\theoremstyle{definition}
\newtheorem{definition}{Definition}

\theoremstyle{remark}
\newtheorem*{remark*}{Remark}

\theoremstyle{definition}
\newtheorem*{notation*}{Notation}

\numberwithin{equation}{section}
\numberwithin{figure}{section}
\numberwithin{theorem}{section}

\title{Nilsequences and multiple correlations along subsequences}

\author{Anh Le}

\address{Department of Mathematics\\
  Northwestern University\\
  2033 Sheridan Road, Evanston, IL 60208-2730, USA}
  
\email{anhle@math.northwestern.edu}

\begin{document}
\begin{abstract}
	The results of Bergelson-Host-Kra and Leibman say that a multiple polynomial correlation sequence can be decomposed into a sum of a nilsequence (a sequence defined by evaluating a continuous function along an orbit in a nilsystem) and a null sequence (a sequence that goes to zero in density). We refine their results by proving that the null sequence goes to zero in density along polynomials evaluated at primes and Hardy sequence $(\lfloor n^c \rfloor)$. On the other hand, given a rigid sequence, we construct an example of correlation whose null sequence does not approach zero in density along that rigid sequence. As a corollary of a lemma in the proof, the formula for the pointwise ergodic average along polynomials of primes in a nilsystem is also obtained. 
\end{abstract}

\maketitle

\section{Introduction}
\label{section:introduction}
    \subsection*{History and motivation}
    
    Let $(X, \mu, T)$ be an invertible measure preserving system, $f_j \in L^{\infty}(\mu)$ and $s_j$ be an \emph{integer polynomial}, i.e. taking integer values on integers, for $0 \leq j \leq k$. Then the sequence
    \begin{equation}
        \label{equation:definition-of-correlation}
        a(n) = \int_X f_0 (T^{s_0(n)} x) \cdot f_1 (T^{s_1(n)} x) \cdots f_k (T^{s_k(n)} x) \, d \mu(x)
    \end{equation}
    is called a \emph{multiple polynomial correlation sequence}, or \emph{polynomial correlation} for short. If $s_j(n) = c_j n$ with $c_j \in \mathbb{Z}$, we call $(a(n))$ a \emph{linear correlation}. 
    
    Understanding multiple correlations has been a main goal of ergodic theorists since Furstenberg's celebrated proof of Szemer\'edi theorem. A possible way is to find connections between correlations and the sequences that have rich algebraic structures. For example, to prove generalized Khintchine's theorem, Bergelson, Host and Kra \cite{Bergelson_Host_Kra05} decompose linear correlations into sum of a \emph{nilsequence} (a sequence defined by evaluating a continuous function along an orbit in a nilsystem) and a \emph{null sequence} (a sequence that approaches zero in density) (see Section \ref{subsec:nilseq} for precise definition).
    
    This decomposition for single linear correlations ($k=1,s_j(n) = c_j n$) can be proved using Herglotz's theorem. In this case, there exists a measure $\sigma$ on the circle $\mathbb{T}$ such that $a(n) = \int_{\mathbb{T}} e^{2 \pi i n x} \, d \sigma(x)$. Decomposing $\sigma$ into discrete (atomic) and continuous (non-atomic) parts, $(a(n))$ is then a sum of an \emph{almost periodic sequence} ($1$-step nilsequence) and a null sequence. 
    
    Bergelson, Host, and Kra \cite{Bergelson_Host_Kra05} extend this classical result to $k \geq 2$ when $s_j(n) = jn$ and $(X, \mu, T)$ being ergodic. In their result, the almost periodic sequence is replaced by a \emph{$k$-step nilsequence}. By a different method, Leibman generalizes Bergelson, Host and Kra's result to the case that $s_j(n)$ are integer polynomials \cite{leib10}. Leibman himself later removes the ergodicity assumption in \cite{leib15}. 
    
    If a sequence $(a(n))$ can be decomposed into a sum of a nilsequence and a null sequence, we say $(a(n))$ has a \emph{nil+null decomposition}. In this case, the decomposition is unique (see Section \ref{subsection:backgroun-uniqueness-of-nil-null-decomposition}). The nilsequence and null sequence are then called the \emph{nil component} and the \emph{null component} of $(a(n))$, respectively.
    
    Nilsequences in general have been studied extensively since their introduction by Bergelson, Host and Kra \cite{Bergelson_Host_Kra05} in 2005. Similarly, the nil components in the nil+null decomposition of multiple correlations are also well studied. For example, Bergelson, Host and Kra \cite{Bergelson_Host_Kra05} analyze this component to prove a generalization of Khintchine's theorem. Moreira and Richter \cite{Moreira_Richter18} show that this component arises from a system whose spectrum is contained in the spectrum of the original system. 
    
    On the other hand, little is known about the null component. The goal of this paper is to fill in that gap. We show that the null component goes to zero in density along polynomials evaluated at primes and Hardy sequence $(\lfloor n^c \rfloor)$. Nevertheless, for any \emph{rigid sequence} $(r_n)$, there is a correlation whose null component is not null along $(r_n)$ (see Section \ref{subsection:backgroun-rigidity-sequences} for definition of rigid sequences).
    
    It is worth mentioning that a related conjecture has been raised by Frantzikinakis \cite[Problem 13]{frant17}. Letting $p_n$ denote the $n^{th}$ prime, Frantzikinakis conjectures that for a linear correlation in an ergodic system $(a(n))$, there exists a nilsequence $(\psi(n))$ and null sequence $(\epsilon(n))$ such that $a(p_n) = \psi(p_n) + \epsilon(n)$. The same conjecture is raised for Hardy sequence $(\lfloor n^c \rfloor)$ instead of $(p_n)$. 
    
    Our result not only gives an affirmative answer to Frantzikinakis' conjecture, but also is stronger in several senses. First, in the case of prime sequences, we work with polynomial correlations rather than linear correlations. Also we do not need the system to be ergodic. Moreover, instead of having different nilsequences when decomposing along $(p_n)$ and $(\lfloor n^c \rfloor)$, we show that there is a fixed nilsequence that works for both, and in fact for many others.
    
    Before presenting the formal statement, we have a definition.
    
    \begin{definition}
        \begin{enumerate}
            \item Let $(r_n)$ be a increasing sequence of integers. A bounded sequence $(a(n))$ is called a \emph{null sequence along $(r_n)$} if 
            \[
                \lim_{N \rightarrow \infty} \frac{1}{N} \sum_{n=1}^N |a(r_n)| = 0
            \]  
            
            \item If $r_n = n$, we simply call $(a(n))$ a \emph{null sequence}.
        \end{enumerate}
    \end{definition}
    
    \subsection*{Statement of results}
    
    The main goal of this paper is to prove:
    
    \begin{theorem}
        \label{theorem:1}
        \begin{itemize}
            \item The null component of a polynomial correlation is null along every sequence of the form $(Q(n))$ or $(Q(p_n))$ where $Q \in \mathbb{Z}[n]$ non-constant and $p_n$ is the $n^{th}$ prime.
            
            \item The null component of a linear correlation is also null along $(\lfloor n^c \rfloor)$ for $c > 0$.
        \end{itemize}
    \end{theorem}
    
\begin{remark*}
    By a different method, Tao and Ter{\"a}v{\"a}inen \cite{tao-teravainen17} prove the null component of a linear correlation is null along the primes, and use this result to prove odd cases of logarithmic Chowla's conjecture.
\end{remark*}
    
    In fact, we prove the null component is null along a more general category of sequences, namely \emph{good sequences}. A good sequence is the one that possesses two properties: \emph{Good for projection on nilfactors (GPN)} and \emph{essentially good for equidistribution on nilmanifolds (EGEN)} (see Section \ref{subsection:good-sequences} for definition). 
    
    To show the null component is null along good sequences, we follow a similar argument as Leibman \cite{leib15}. A key proposition in Leibman's proof says that an \emph{integral of nilsequences} has nil+null decomposition (see Section \ref{subsection:integral-of-nilsequences}). In Section \ref{section:integral-of-nilsequence}, we refine that result by showing that:
    \begin{proposition}
        \label{proposition:integral-of-nilsequences}
        The null component of an integral of nilsequences is null along any EGEN sequence.
    \end{proposition}

    The fact that $(Q(n))$ is a good sequence follows from the works of Host-Kra \cite{hk05-2}, and Leibman \cite{leib05-3, leib05}. On the other hand, Hardy sequence $(\lfloor n^c \rfloor)$ is proved to be good by Frantzikinakis \cite{frant09, frant10}. In Section \ref{section:polynomial-of-primes-is-good}, we show:
    
    \begin{proposition}
        \label{proposition:poly-prime-good}
        For any $Q \in \mathbb{Z}[n]$ non-constant, the sequence $(Q(p_n))$ is good.
    \end{proposition}
    
    The EGEN property of polynomials of primes allows us to determine the exact formula for pointwise ergodic average along polynomials of primes for continuous functions in a \emph{nilsystem} (see Section \ref{subsec:nilseq} for definition). Green and Tao \cite{gt12} prove the average converges to the integral of the function in case of a totally ergodic nilsystem. Eisner \cite{eis} shows that the average converges everywhere in an arbitrary nilsystem. But the exact formula is still missing in this general situation. As a corollary of the EGEN property of $(Q(p_n))$, we can determine the exact average. 
    
    To be precise, for an ergodic nilsystem $(X = G/\Gamma, \mu, \tau)$, let $\pi: G \to X$ be the canonical map $\pi(g) = g \Gamma$. Assume X has $d$ connected components, and $X_0$ is the component containing $1_X = \pi(1_G)$. Let $X_j = \tau^j X_0$ for $j \in \mathbb{Z}$ and $\mu_{X_j}$ be the Haar measure of $X_j$. Note that $X_i = X_j$ if $i \equiv j \pmod d$ (See Section \ref{subsubsec:connected-nilorbits}). Let $\phi$ be the Euler function. Then we have:
    \begin{corollary}
    \label{corolarry:average-along-prime-formula}
        Let $(X, \mu, \tau)$ be an ergodic nilsystem with $d$ connected components $X_0, X_1, \ldots, X_{d-1}$ with $X_i = \tau^i X_0$ and $f$ be a continuous function on $X$. Suppose $x \in X_k$ for some $0 \leq k \leq d-1$. Then
        \[
            \lim_{N \to \infty} \frac{1}{N} \sum_{n=1}^N f(\tau^{Q(p_n)} x) = \frac{1}{\phi(d)} \sum_{\substack{1 \leq s < d \\ (s,d) = 1}} \int_{X_{Q(s) + k}} f \, d \mu_{X_{Q(s) +k}}        
        \]
    \end{corollary}

    In the same spirit of Theorem \ref{theorem:1}, but in the opposite direction, we are also interested in those sequences $(r_n)$ such that there exists a correlation whose null component is not null along $(r_n)$. It turns out there is a well known class of sequences satisfying such condition, namely \emph{rigid sequences}. A sequence is called \emph{rigid} if there is a weakly mixing system $(X, \mu, T)$ such that $\lVert T^{r_n} f - f \rVert_{L^2(\mu)} \to 0$ for all $f \in L^2(\mu)$. Examples of rigid sequences include $(2^n)$, $(3^n)$ and $(n!)$ (See Section \ref{subsection:backgroun-rigidity-sequences} for more details). In Section \ref{section:not-null-along-3^n}, we prove:    
    \begin{proposition}
        \label{proposition:3^n}
        For any rigid sequence $(r_n)$, there exists a linear correlation whose null component is not null along $(r_n)$.
    \end{proposition}
    
    \subsection*{ Application}
    
    The goal of Bergelson, Host and Kra's paper \cite{Bergelson_Host_Kra05} is not to prove the nil+null decomposition for a multiple correlation. They use the decomposition to prove a generalization of Khintchine's Theorem. In a similar fashion, it follows from Theorem \ref{theorem:1} and Corollary \ref{corolarry:average-along-prime-formula} that in an ergodic system $(X, \mu, T)$, for any measurable set $A \subseteq X$ and $\delta > 0$, the set
    \[
        \{n \in \mathbb{N}: \mu(A \cap T^{-(p_n-1)} A \cap T^{-2(p_n -1)} A) \geq \mu(A)^3 - \delta\}
    \]
    has positive density. The same is true for the set
    \[
        \{n \in \mathbb{N}: \mu(A \cap T^{-(p_n -1)} A \cap T^{-2(p_n -1)} A \cap T^{-3(p_n -1)} A) \geq \mu(A)^4 - \delta\}.
    \]
    A detail proof will appear in a forthcoming paper \cite{Donoso-Le-Moreira-Sun-2018}.
    
    \subsection*{Open question}
    It is still open that whether a similar result to nil+null decomposition exists for a set of commuting transformations. To be precise, for a measure space $(X, \mu)$ with commuting measure preserving transformations $T_j \colon X \to X$ and $f_j \in L^{\infty}(\mu)$ for $0 \leq j \leq k$, we define a correlation sequence
    \[
        a(n) = \int_X f_0(T_0^{n} x) f_1 (T_1^{n} x) \ldots f_k(T_k^{n} x) \, d \mu(x).
    \]
    Frantzikinakis \cite{frant15} shows that for any $\delta > 0$ the sequence $(a(n))$ can be decomposed as $a(n) = a_{st}(n) + a_{er}(n)$ where $(a_{st}(n))$ is a $k$-step nilsequence and 
    \[
        \lim_{N - M \rightarrow \infty} \frac{1}{N-M}  \sum_{n = M}^{N-1} |a_{er}(n)|^2 < \delta.
    \]
    From Frantzikinakis' result, it is natural to ask whether we have the same decomposition, but in addition 
    \[
        \lim_{N \rightarrow \infty} \frac{1}{N} \sum_{n = 1}^N |a_{er}(p_n)|^2 < \delta?
    \]
    What if we replaced $(p_n)$ by  Hardy sequence $(\lfloor n^c \rfloor)$? Our argument in this paper does not apply since we do not have sufficient information about the factors that control the multiple ergodic averages for commuting transformations \cite{austin15-1, austin15-2}. These factors are not simply inverse limits of nilsystems, the objects that play a crucial role in our analysis.
    
    \subsection*{Outline of the paper} 
    Section \ref{sec:background} is for background and notation. In Section \ref{section:integral-of-nilsequence}, we prove the Proposition \ref{proposition:integral-of-nilsequences} about the integral of nilsequences. In Section \ref{section:strong-nil-null-decomposition}, we proceed to prove the null component of a correlation is null along good sequences. Section \ref{section:polynomial-of-primes-is-good} is to show the polynomials of primes are good sequences, hence effectively prove Theorem \ref{theorem:1}. Also in this section, we prove the limit formula of the average along polynomials of primes (Corollary \ref{corolarry:average-along-prime-formula}). In the last section, we construct an example of correlations whose null component is not null along a given rigid sequence. 
    
    \subsection*{Acknowledgment}
    I would like to thank B. Kra, N. Frantzikinakis and J. Moreira for many valuable advices. I also would like to thank A. Leibman for answering my questions about his papers and the anonymous referee for suggestions that vastly improve the readability of this paper.
    
\section{Background and notation}
\label{sec:background}

    \subsection{Notation}
        A sequence is a function $a \colon \mathbb{N} \rightarrow \mathbb{C}$. We denote this sequence by $(a(n))_{n \in \mathbb{N}}$, $(a(n))$ or sometimes only $a$ if there is no danger of confusion.
        
        For $N \in \mathbb{N}$, we write $[N] = \{1, 2, \ldots, N\}$. For a function $f$ on a finite non-empty set $S$, let $\mathbb{E}_{s \in S} f(s)$ denote $\frac{1}{|S|} \sum_{s \in S} f(s)$. In particular, for bounded sequence $(a(n))$,  
        \[
            \mathbb{E}_{n \in [N]} a(n) := \frac{1}{N} \sum_{n = 1}^N a(n) 
        \]
        
        Let $(X, \mu, T)$ be a measure preserving system and $f \in L^{\infty}(X)$. $Tf$ is defined to be $Tf(x) := f(Tx)$ for all $x \in X$. If $(Y, \nu, S)$ is a factor of $(X, \mu, T)$, we denote \emph{conditional expectation of $f$ on $Y$} by $\mathbb{E}(f|Y)$.
        
        Let $\mathbb{P}$ denote the set of all primes and $p_n$ the $n^{th}$ prime. For $d, s \in \mathbb{Z}$, let $p_{s \pmod d, n}$ be the $n^{th}$ prime that is congruent to $s \pmod d$. 
        
    \subsection{Nilmanifolds, nilsystems and nilsequences}
    \label{subsec:nilseq}
        Let $G$ be a $k$-step nilpotent Lie group and $\Gamma$ be a \emph{uniform} (i.e closed and cocompact) subgroup of $G$. The compact homogeneous space $X := G/\Gamma$ is called a \emph{$k$-step nilmanifold}. Let $\pi \colon G \rightarrow X$ be the standard quotient map. We write $1_X = \pi(1_G)$ where $1_G$ is the identity element of $G$. Suppose $G^0$ is the identity connected component of $G$. If $X$ is connected, then $X = \pi(G^0) = G^0/(G^0 \cap \Gamma$). 
        
        The space $X$ is endowed with a unique probability measure that is invariant under the translations by $G$. This measure is called the \emph{Haar measure} for $X$, and denoted by $\mu_X$. For every $\tau \in G$, the measure preserving system $(X, \mu_X, \tau)$ is called \emph{$k$-step nilsystem}.
        
        Let $C(X)$ denote the set of continuous functions on $X$. For $f \in C(X)$ and $x \in X$, the sequence $\psi(n) := f(\tau^n x)$ is called a \emph{basic $k$-step nilsequence}. A \emph{$k$-step nilsequence} is a uniform limit of basic $k$-step nilsequences.
        
        If $G$ is not connected, we can embed $X$ in $X' = G'/\Gamma'$ where $G'$ is a connected and simply-connected $k$-step nilpotent Lie group and $\Gamma'$ is a closed, discrete cocompact subgroup of $G'$. Extending $f$ to a continuous function $f'$ on $X'$ and suppose $\tau' \in G'$ and $x' \in X'$ are elements corresponding to $\tau \in G$ and $x \in X$, we have a different representation of basic $k$-step nilsequence $\psi(n) =  f'( \tau'^n x' )$ for all $n \in \mathbb{Z}$. Therefore, if the basic $k$-step nilsequence $(f(\tau^n x))_{n \in \mathbb{Z}}$ is our interest, without the loss of generality, we can assume $G$ is connected and simply connected. 

        \begin{remark*}
            Different authors may have different notion of nilsequences. We use the original definition by Bergelson, Host and Kra \cite{Bergelson_Host_Kra05}. Leibman in his series of papers \cite{leib10, leib15} uses the same definition. However, in Green, Tao \cite{gt10, gt12} and Frantzikinakis \cite{frant15}, the nilsequences are in fact our basic nilsequences. Frantzikinakis \cite{frant17} even introduces the notion of \emph{basic generalized $k$-step nilsequences}. They are sequences of the form $(f(\tau^n x))_{n \in \mathbb{N}}$ when $f$ is allowed to be Riemann integrable. We do not use this notion in current paper.
        \end{remark*}
        
    \subsection{Subnilmanifolds}
        Let $X = G/\Gamma$ be a $k$-step nilmanifold. A \emph{subnilmanifold} $Y$ of $X$ is a closed subset of $X$ of the form $Y = Hx$ where $H$ is a closed subgroup of $G$ and $x \in X$. The Haar measure on $Y$ is denoted by $\mu_Y$. This measure is invariant under translation by any $\tau \in G$. 
        
        A \emph{normal subnilmanifold} $Z$ of $X$ is a subnilmanifold which is equal to $Lx$ for some normal closed subgroup $L$ of $G$ and $x \in X$. The quotient nilmanifold $X/Z := G/(L \Gamma)$ is a factor of $X$ by standard factor map $G/\Gamma \rightarrow G/(L \Gamma)$. For a subnilmanifold $Y$ of $X$, the \emph{normal closure} of $Y$ in $X$ is the smallest normal subnilmanifold of $X$ that contains $Y$. The normal closure of a connected subnilmanifold is connected \cite[page 5]{leib15}.
        
        For $\tau \in G$, we say the sequence $(\tau^n Y)_{n \in \mathbb{N}}$ is \emph{equidistributed} on $X$ if for any $f \in C(X)$, 
        \[
            \lim_{N \rightarrow \infty} \mathbb{E}_{n \in [N]} \int_Y \tau^n f \, d\mu_Y = \int_X f \, d \mu_X.
        \]
        
    \subsection{Orbit closures of subnilmanifolds}
        In this section we summarize important facts about orbit closures of subnilmanifolds under linear and polynomial translations. 
        
        \subsubsection{Linear orbits}
        \label{subsubsec:connected-nilorbits} 
        Let $Y$ be a connected subnilmanifold of nilmanifold $X = G/\Gamma$ and $\tau \in G$. Then the orbit closure of $Y$ under action of $\tau$ is a subnilmanifold of $X$, namely $\overline{\{\tau^n Y\}}_{n \in \mathbb{N}}$ and is denoted by $\mathcal{O}_Y$. Suppose $d$ is the number of connected components of $\mathcal{O}_Y$ and $Y^0$ is the component containing $Y$. Then all connected components of $\mathcal{O}_Y$ are $Y^0, \tau Y^0, \tau^2 Y^0, \ldots, \tau^{d-1} Y^0$. Moreover $\tau^{dn+r} Y^0 = \tau^r Y^0$ for $n \in \mathbb{N}, r \in \mathbb{Z}$ and the sequence $(\tau^{dn+r} Y)_{n \in \mathbb{N}}$ is equidistributed in $\tau^r Y^0$.
        
        In particular, suppose $(X, \mu, \tau)$ is an ergodic nilsystem with $d$ connected components. Assume $X_0$ is the component containing $1_X = \pi(1_G)$. Then all components of $X$ are $X_0, \tau X_0, \ldots, \tau^{d-1} X_0$. And $(\tau^{dn+r} 1_X)_{n \in \mathbb{N}}$ is equidistributed on $\tau^r X_0$. For details and proofs, see \cite{leib05}.
        
        \subsubsection{Polynomial orbits} A nilsystem $(X, \mu, \tau)$ is totally ergodic if and only if $X$ is connected \cite[Proposition 2.1]{Frantzikinakis08}. In this case, for any $Q(n) \in \mathbb{Z}[n]$ non-constant, and $x \in X$, the sequence $(\tau^{Q(n)} x)$ is equidistributed on $X$. A stronger result is obtained in \cite[Lemma 6.7]{frant10} 
        
    \subsection{Uniqueness of nil+null decomposition}
    \label{subsection:backgroun-uniqueness-of-nil-null-decomposition}
    
    If a sequence $(a(n))$ have two nil+null decompositions $a = \psi_1 + \epsilon_1 = \psi_2 + \epsilon_2$ where $\psi_1, \psi_2$ are nilsequences and $\epsilon_1, \epsilon_2$ are null sequences. Then $\psi_1 - \psi_2 = \epsilon_2 - \epsilon_1$. 
    
    $\psi_1 - \psi_2$ is a nilsequence and $\epsilon_2 - \epsilon_1$ is a null sequence. A nilsequence returns to any neighborhood of its supremum in a bounded gap set (due to minimality of an ergodic nilsystem). Hence it is a null sequence only when the supremum is $0$. Thus in our case, $\psi_1 - \psi_2 = \epsilon_2 - \epsilon_1 = 0$.
    
    \subsection{Integral of nilsequences}
    \label{subsection:integral-of-nilsequences}
        Let $(\Omega, \rho)$ be a measure space. Suppose for each $\omega \in \Omega$, there is a nilsequence $(\psi_{\omega}(n))_{n \in \mathbb{Z}}$. We say the family of nilsequences $\{\psi_{\omega}: \omega \in \Omega\}$ is \emph{integrable with respect to $\rho$} if for each $n \in \mathbb{Z}$, the function $\omega \mapsto \psi_{\omega}(n)$ is integrable with respect to $\rho$. In this case, the sequence $a(n) = \int_{\Omega} \psi_{\omega}(n) \, d \rho(\omega)$ is called \emph{an integral of nilsequences}. Leibman \cite[Proposition 4.2]{leib15} proves an integral of nilsequences admits a nil+null decomposition.
        
    \subsection{Nilfactors}
        Let $(X, \mu, T)$ be an ergodic measure preserving system. Suppose $(s_j(n))_{n \in \mathbb{N}}$ is integer valued sequence for $1 \leq j \leq k$. A factor $(Y, \nu, S)$ of $\textbf{X}$ is said to be \emph{characteristic for $(s_1(n), \ldots, s_k(n))$} if for any bounded functions $f_1, ..., f_k$ on $X$, we have 
        \[
            \lim_{N \rightarrow \infty} \left( \mathbb{E}_{n \in [N]} \prod_{j=1}^k T^{s_j(n)} f_j - \mathbb{E}_{n \in [N]} \prod_{j=1}^k T^{s_j(n)} \mathbb{E}\left(f_j|Y \right) \right) = 0,
        \]
        where the limits are taken in $L^2(X, \mu)$. Host and Kra \cite{hk05} show that there exists a characteristic factor for $(n,2n, \ldots, kn)$ which is an inverse limit of $(k-1)$-step nilsystems. We call this factor the \emph{$(k-1)$-step nilfactor of $X$} and denote it by $\mathcal{Z}_{k-1}(X)$ (some time $\mathcal{Z}_{k-1}$ if there is no confusion).
        
        Host and Kra \cite{hk05-2} show for most families of integer polynomials $Q_j$, there exists a nilfactor $\mathcal{Z}_m$ that is characteristic for $(Q_1(n), \ldots, Q_k(n))$. Leibman \cite{leib05-3} later show that result is true for all families of integer polynomials.

        \subsection{Characteristic factor for integer polynomials of primes}
        \label{subsec:char-pol-prime}
        Frantzikinakis, Host and Kra \cite{fhk} prove that $\mathcal{Z}_1$ factor is characteristic for $2$-tuple $(p_n, 2 p_n)$ where $p_n$ is the $n^{th}$ prime. For $k \geq 3$, they show that $\mathcal{Z}_{k-1}$ is characteristic for $k$-tuple $(p_n, 2p_n, \ldots, k p_n)$ conditional upon results on Mobius function and inverse conjecture for the Gowers norms, which are now established by Green and Tao \cite{gt10} and Green, Tao, and Ziegler \cite{gtz} respectively. 

    \subsection{Relative products}
    Let $X_1$, $X_2$ and $Y$ be three sets. Suppose there are surjective maps $\delta_1: X_1 \rightarrow Y$ and $\delta_2: X_2 \rightarrow Y$. Then \emph{fiber product of $X_1$ and $X_2$} with respect to $Y$ is defined to be $\{(x_1, x_2) \in X_1 \times X_2: \delta_1(x_1) = \delta_2(x_2)\}$. We denote this product by $X_1 \times_Y X_2$.
    
    Suppose $(X_1, \mu_1, T_1)$ and $(X_2, \mu_2, T_2)$ are measure preserving systems. Let $(Y, \nu, S)$ be a common factor of $(X_1, \mu_1, T_1)$ and $(X_2, \mu_2, T_2)$. Then the \emph{relative product} of $X_1$ and $X_2$ with respect to $Y$ is the measure preserving system $(X_1 \times_Y X_2, \mu_1 \times_Y \mu_2, T_1 \times T_2)$ where: 
    \begin{enumerate}[label=(\roman*)]
        \item The space $X_1 \times_Y X_2$ is the fiber product of $X_1$ and $X_2$ with respect to $Y$
        
        \item The measure $\mu_1 \times_Y \mu_2$ is characterized by 
        \[
            \int_{X_1 \times_Y X_2} f_1 (x_1) \otimes f_2(x_2) \, d (\mu_1 \times_Y \mu_2) (x_1, x_2) = \int_Y \mathbb{E}(f_1|Y) \mathbb{E}(f_2|Y) \, d \nu
        \]
        for all $f_1 \in L^2(X_1)$ and $f_2 \in L^2(X_2)$.
    \end{enumerate} 
    By abusing of notation, let $X_1 \times_{Y} X_2$ denote the relative product of $X_1$ and $X_2$ with respect to $Y$. If $X_1$ and $X_2$ are nilsystems, and $Y$ is common nilsystem factor, then $X_1 \times_Y X_2$ is also a nilsystem.
    
    \subsection{Hardy sequences}
        Let $\mathcal{F}$ be the collection of a functions $f \colon \mathbb{R}_{>0} \to \mathbb{R}$. Define $\mathcal{B} = \mathcal{F}/\sim$ where $f \sim g$ if there exists constant $c > 0$ such that $f(x) = g(x)$ for all $x > c$. A \emph{Hardy field} is a subfield of the ring $(\mathcal{B}, +, \times)$ which is closed under differentiation. An example of Hardy fields is the set of functions that are combinations of addition, multiplication, exponential and logarithm on real variable $t$ and real constants. Let $\mathcal{H}$ be the union of all Hardy fields. 
        
        For $a, b \in \mathcal{H}$, we write $a(t) \succ b(t)$ if $\lim_{t \rightarrow \infty} b(t)/a(t) = 0$. We say a function $a(t)$ has polynomial growth if there exists a polynomial $p \in \mathbb{R}[t]$ such that $p(t) \succ a(t)$. We call the sequence $(\lfloor a(n) \rfloor)_{n \in \mathbb{N}}$ a \emph{Hardy sequence} where $a \in \mathcal{H} $ and $\lfloor .\rfloor$ indicates integral part.

        \begin{definition}
            Let $a \in \mathcal{H}$ have polynomial growth and satisfy $a(t) - cp(t) \succ \log t$ for every $c \in \mathbb{R}$ and $p \in \mathbb{Z}[t]$. Then the sequence $(\lfloor a(n) \rfloor)_{n \in \mathbb{N}}$ is called \emph{a Hardy sequence of polynomial growth and logarithmically away from every multiple of polynomial of integer coefficients}.
        \end{definition}
        
        Examples of sequences that satisfy previous definition are $(\lfloor n^{c} \rfloor)_{n \in \mathbb{N}}$ where $c > 0$, $c \not \in \mathbb{Z}$, $(\lfloor n \log n \rfloor)_{n \in \mathbb{N}}$, $(n^2 \sqrt{2} + n \sqrt{3})_{n \in \mathbb{N}}$, $(n^3 + (\log n)^3)_{n \in \mathbb{N}}$. From now on, whenever we write $(\lfloor n^c \rfloor)$, it represents the entire class of Hardy sequences of polynomial growth and logarithmically away from every multiple of polynomial of integer coefficients.
    \subsection{Good sequences}
    \label{subsection:good-sequences}
    
        \begin{definition}
            \begin{enumerate}
                \item The sequence $(r_n)_{n \in \mathbb{N}}$ is said to be \emph{linearly good for projection onto nilfactors} (denoted by \emph{linear-GPN}) if for any $h_1, h_2, \ldots, h_k \in \mathbb{Z}$, there is some $m$ such that $m$-step nilfactor is characteristic for $(h_1 r_n, h_2 r_n, \ldots, h_k r_n)$.
            
                \item Similarly, $(r_n)_{n \in \mathbb{N}}$ is said to be \emph{polynomially good for projection onto nilfactors} (\emph{polynomial-GPN}) if for any $s_1, s_2, \dots, s_k \in \mathbb{Z}[n]$, there is some $m$ such that $m$-step nilfactor is characteristic for $(s_1(r_n), s_2(r_n), \dots, s_k(r_n))$. It is obvious that a polynomial-GPN sequence is linear-GPN. 
            \end{enumerate}
        \end{definition}
        
        By the works of Host-Kra \cite{hk05-2} and Leibman \cite{leib05-2}, polynomial sequence $(Q(n))$ is polynomial-GPN. On the other hand, Frantzikinakis \cite{frant10} shows that $(\lfloor n^c \rfloor)$ with $c > 0, c \not \in \mathbb{Z}$ is linear-GPN. 
        
        \begin{definition}
            \begin{enumerate}
                \item The sequence $(r_n)$ is said to be \emph{good for equidistribution on nilmanifolds} (denoted by \emph{GEN}) if for an ergodic nilsystem $(X, \mu, \tau)$, the sequence $(\tau^{r_n} 1_X)_{n \in \mathbb{N}}$ is equidistributed on $X$.
            
                \item The sequence $(r_n)$ is called \emph{essentially good for equidistribution on nilmanifolds (EGEN)} if the following holds: Suppose for some $s, d \in \mathbb{N}$ such that the set $\{n \in \mathbb{N}: r_n \equiv s \pmod d\}$ has positive upper density. Let $r_{s \pmod d, n}$ denote the $n^{th}$ element of $\{ r_m: m \in \mathbb{N} \}$ that is congruent to $s \pmod d$. Let $(X, \mu, \tau)$ be an ergodic nilsystem with $d$ connected components and $X_0$ be the component containing $1_X$. Then the sequence $(\tau^{r_{s \pmod d, n}} 1_X)_{n \in \mathbb{N}}$ is equiditributed on $\tau^s X_0$.
            \end{enumerate}
        \end{definition}

\begin{remark*}
Here is the difference between GEN and EGEN. In an ergodic nilsystem, the orbit of any point along a GEN sequence is equidistributed on the on the nilmanifold. On the other hand, the orbit along an EGEN sequence may not be. However, if we restrict to a suitable arithmetic progression, the orbit now is equidistributed on a connected component of the nilmanifold.
                
It is easy to see that a GEN sequence is EGEN. 
\end{remark*}

        Frantzikinakis \cite{frant09} proves that $(\lfloor n^c \rfloor)$ with $c > 0$, $c \not \in \mathbb{Z}$ is GEN. He also proves polynomial sequences are EGEN \cite[Lemma 6.7]{frant10}. To demonstrate why polynomials satisfy EGEN property but not GEN, take for example $Q(n) = n^2$, $X = \mathbb{T} \times \mathbb{Z}/3$ and $\tau = (\alpha, \bar{1})$ where $\alpha$ is irrational. Then since $n^2 \equiv 0$ or $1 \pmod 3$, the sequence $(\tau^{n^2} (0,0))$ never visits the connected component $\mathbb{T} \times \bar{2}$. So it is not equidistributed on entire $\mathbb{T} \times \mathbb{Z}/3$. However, if we restrict to those $n \equiv 0 \pmod 3$, i.e $n = 3m$ then the sequence $(\tau^{((3m)^2}(0,0)$ is now equidistributed on the component $\mathbb{T} \times \bar{0}$. Similarly the sequences $(\tau^{(3m+1)^2}(0,0))$ and $(\tau^{(3m+2)^2}(0,0))$ are equidistributed on $\mathbb{T} \times \bar{1}$.
        
        \begin{definition}
            A sequence that is both linear-GPN and EGEN is called a \emph{linear-good sequence}. Analogously, a sequence that is both polynomial-GPN and EGEN is called \emph{polynomial-good sequence}.
        \end{definition}
        
        From above discussion, we see that polynomial sequence $(Q(n))$ is polynomial-good while Hardy sequence $(\lfloor n^c \rfloor)$ is linear-good.

\subsection{Gaussian system}
\label{subsection:Gaussian-system}
For a positive measure $\sigma$ on $\mathbb{T}$, there exists a Gaussian system $(X, \mu, T)$ and function $g \in L^2(\mu)$ such that $\hat{\sigma}(n) = \int_X g T^n \bar{g} \, d \mu$ for all $n \in \mathbb{N}$. If $\sigma$ is a probability measure, then $\lVert g \rVert_{L^2(\mu)} = 1$. It is worth to mention that $g$ is a Gaussian variable, hence is unbounded. See \cite[pages 369-371]{Cornfeld_Fomin_Sinai82} for details. 

\subsection{Rigid sequences}
	\label{subsection:backgroun-rigidity-sequences}
	We recall the defintion of rigid sequences from the introduction. An increasing sequence of integers $(r_n)$ is called rigid if there is a weakly mixing system $(X, \mu, T)$ such that $\lVert T^{r_n} f - f \rVert_{L^2(\mu)} \to 0$ for all $f \in L^2(\mu)$. Using Gaussian systems, we can show that a sequence $(r_n)$ is rigid if and only if there is a continuous measure $\sigma$ on $\mathbb{T}$ such that $\hat{\sigma}(r_n) \to 1$ as $n \to \infty$. 
    
    Examples of rigid sequences include $(q^n)_{n \in \mathbb{N}}$ for $q \in \mathbb{N}$, $q \geq 2$. Generally, an increasing sequence $(r_n)$ such that $r_n | r_{n+1}$ is rigid \cite{bergelson_deljunco_lemanczyk_rosenblatt_2014,Eisner-Grivaux-2011}. Furtheremore, there is rigid sequence with very slow growth. Let $(d_n)$ be an increasing sequence of integers of density zero. Then there is a rigid sequence $(r_n)$ such that $r_n \leq d_n$ for all $n \in \mathbb{N}$ \cite{Aaronson1979}. See \cite{bergelson_deljunco_lemanczyk_rosenblatt_2014}, \cite{Eisner-Grivaux-2011}, \cite{Badea-Grivaux2018} and \cite{Fayad-Thouvenot-2014} for more exhaustive lists of rigid sequences.        
    
\section{Integral of nilsequences}
    \label{section:integral-of-nilsequence}
    
To prove a correlation sequence in a non-ergodic system has a nil+null decomposition, Leibman \cite{leib15} shows that an integral of nilsequences has such decomposition. For this purpose, by a series of reduction, Leibman proves that it suffices to show:
\begin{proposition}[Leibman {\cite[Proposition 4.3]{leib15}}]
    \label{proposition:reduction-to-integration-on-nilgroup}
    Let $X = G/\Gamma$ be a nilmanifold, $\rho$ be a finite Borel measure on $G$ and $f \in C(X)$, then the sequence $\varphi(n) = \int_G f(g^n 1_X) \, d \rho(g)$ has a nil+null decomposition.
\end{proposition}

By the same reduction, for the purpose of showing Proposition \ref{proposition:integral-of-nilsequences}, i.e. the null component of an integral of nilsequences is null along EGEN sequences, it suffices to show:
\begin{proposition}
    \label{proposition:reduction-to-integration-on-nilgroup2}
    With the set-up as in Proposition \ref{proposition:reduction-to-integration-on-nilgroup}, in the nil+null decomposition of $(\varphi(n))$, the null component is null along any EGEN sequence.
\end{proposition}
    
The rest of this section is devoted to prove Proposition \ref{proposition:reduction-to-integration-on-nilgroup2}. We start with a lemma.

\begin{lemma}
    \label{lemma:main-integral-combination}
    Let $X= G/\Gamma$ be a nilmanifold and $Z$ be a normal subnilmanifold that contains $1_X$. Suppose $\tau \in G^0$ such that $(\tau^n Z)_{n \in \mathbb{N}}$ is dense in $X$. Let $\rho$ be a finite Borel measure on $G$ such that for $\tilde{\rho} = \pi_*(\rho)$ we have $\rm{supp} (\tilde{\rho}) \subseteq \tau Z$ and $\tilde{\rho}(\tau W) = 0$ for any proper normal subnilmanifold $W$ of $Z$. Let $\varphi(n) = \int_G f ( g^n 1_X) \, d \rho(g)$ for $n \in \mathbb{N}$, $\hat{X} = X/Z$ and $\hat{f} = \mathbb{E}(f|\hat{X})$. Then $(\varphi(n) - \hat{f}(\pi(\tau ^n)))_{n \in \mathbb{N}}$ is null along any EGEN sequence. In particular, the null component of $\varphi$ is null along any EGEN sequence.
\end{lemma}

\begin{proof}
Let $(r_n)$ be an arbitrary EGEN sequence. Replacing $f$ by $f - \hat{f}$, we can assume $\mathbb{E}(f|\hat{X}) = 0$. We are left with showing $(\varphi(n))_{n \in \mathbb{N}}$ is null along $(r_n)$. 

Shifting $\rho$ to the origin by replacing it by $\tau_*^{-1} \rho$. Let $L$ be a connected subgroup of $G$ such that $\pi(L) = Z$. So now $\rm{supp}(\rho) \subseteq L$ and $\rm{supp}(\tilde{\rho}) \subseteq Z$. 
        
Let $d \in \mathbb{N}$ be the number of connected components of $X \times_{\hat{X}} X$. Note that to show $\varphi$ is null along $(r_n)$, it suffices to show $\varphi$ is null along $(r_{s \pmod d, n})$ for any $0 \leq s \leq d-1$ such that the set $\{n \in \mathbb{N}: r_n \equiv s \pmod d\}$ has positive upper density. Let $s$ be one of such number. Define
\[
    H(a,b) = \lim_{N \rightarrow \infty} \mathbb{E}_{n \in [N]} f \otimes \bar{f} (\pi^{\times 2}((\tau a, \tau b)^{r_{s \pmod d, n}}))
\]
and
\[
    F(a,b) = \lim_{N \rightarrow \infty} \mathbb{E}_{n \in [N]} f \otimes \bar{f} (\pi^{\times 2}((\tau a, \tau b)^{dn + s}))
\]
for $(a, b) \in L \times L$.

According to the proof of Lemma 4.6 in Leibman \cite{leib15}, for $\rho^{\times 2}$-almost every $(a,b) \in L \times L$, the sequence $u_n = (\tau a, \tau b)^n (1_X, 1_X) = \pi^{\times 2}((\tau a, \tau b)^n)$ is equidistributed on $X \times_{\hat{X}} X$. Therefore, for those $(a,b)$, by Section \ref{subsubsec:connected-nilorbits}, the sequences $(\tau a, \tau b)^{dn +s} (1_X, 1_X)$ is equidistributed on $(\tau a, \tau b)^s (X \times_{\hat{X}} X)_o$ where $X \times_{\hat{X}} X)_o$ is the connected component of $X \times_{\hat{X}} X$ containing $(1_X, 1_X)$. 

On the other hand, by definition of EGEN, the sequence $(\tau a, \tau b)^{r_{s \pmod d, n}}(1_X, 1_X)$ is also equidistributed on $(\tau a, \tau b)^s (X \times_{\hat{X}} X)_o$. That implies $H(a,b) = F(a,b) = \int_{(\tau a, \tau b)^s (X \times_{\hat{X}} X)_o} f \otimes \bar{f} \, d \mu_{(\tau a, \tau b)^s (X \times_{\hat{X}} X)_o}$. This equality holds for $\rho^{\times 2}$-almost every $(a,b) \in L \times L$. So by taking integral on $L \times L$, with respect to $\rho^{\times 2}$, we get
\[
    \lim_{N \rightarrow \infty} \mathbb{E}_{n \in \mathbb{N}} |\varphi(dn+s)|^2 = \lim_{N \rightarrow \infty} \mathbb{E}_{n \in \mathbb{N}} |\varphi(r_{s \pmod d,n})|^2.
\]

The sequence $(\varphi(n))_{n \in \mathbb{N}}$ is a null sequence along $(n)_{n \in \mathbb{N}}$. The subsequence $(dn + s)_{n \in \mathbb{N}}$ has density $1/d$ in $\mathbb{N}$. It follows that $(\varphi(n))_{n \in \mathbb{N}}$ is also a null sequence along $(dn+s)_{n \in \mathbb{N}}$. Thus it follows that $\varphi$ is null along $(r_{s \pmod d, n})_{n \in \mathbb{N}}$. This fact holds true for any $0 \leq s \leq d-1$ such that $\{n \in \mathbb{N}: r_n \equiv s \pmod d\}$ has positive upper density. Hence $\varphi$ is null along $(r_n)$. Since $(r_n)$ is an arbitrary EGEN sequence, we have $\varphi$ is null along any EGEN sequence. 

We just show $(\varphi(n) - \hat{f}(\pi(\tau^n)))$ is null along any EGEN sequence. On the other hand, $\hat{f}(\pi(\tau^n))$ is a nilsequence by definition. Thus $(\varphi(n) - \hat{f}(\pi(\tau^n))$ is the null component of $\varphi(n)$, and it is null along any EGEN sequence. This finishes our proof.
\end{proof}

We need a lemma from Leibman \cite{leib15}.

\begin{lemma}[Leibman{\cite[Lemma 4.4]{leib15}}]
\label{lemma:decomposition-of-measure}
    Let $X = G/\Gamma$ be a nilmanifold with standard quotient map $\pi: G \rightarrow X$. Suppose $\rho$ is a finite Borel measure on $G$. Then there exists an at most countable collection $\mathcal{V}$ of connected subnilmanifolds of $X$ and finite Borel measure $\rho_V$ for $V \in \mathcal{V}$ on $G$ such that $\rho = \sum_{V \in \mathcal{V}} \rho_V$ and for every $V \in \mathcal{V}$, $\rm{supp}(\tilde{\rho}_V) \subseteq V$ and $\tilde{\rho}_V(S) = 0$ for any proper subnilmanifold $S$ of $V$ where $\tilde{\rho}_V = \pi_*(\rho_V)$.
\end{lemma}

We are ready to prove Proposition \ref{proposition:reduction-to-integration-on-nilgroup2}.

\noindent\textbf{Proof of Proposition \ref{proposition:reduction-to-integration-on-nilgroup2}.}
    By Lemma \ref{lemma:decomposition-of-measure}, the measure $\rho$ can be decomposed as $\rho = \sum_{V \in \mathcal{V}} \rho_V$ where $\rm{supp} (\tilde{\rho}_V) \subseteq V$ and $\tilde{\rho}_V(S) = 0$ for any proper subnilmanifold $S$ of $V$.
    
    Fix $V \in \mathcal{V}$. Let $V'$ be the normal closure of $V$ in $X$. For any proper normal subnilmanifold $S'$ of $V'$, the intersection $S' \cap V$ is a proper subnilmanifold of $V$ by the minimality of $V'$. Therefore $\tilde{\rho}_V(S' \cap V) = 0$. Since $\rm{supp}(\tilde{\rho}_V) \subseteq V$, we have $\tilde{\rho}_V(S') = \tilde{\rho}_V (S' \cap V) + \tilde{\rho}_V (S' \setminus V) = 0$. 
    
    Write $V' = \tau Z$ for $\tau \in G^0$ and a normal subnilmanifold $Z$ of $X$ that contains $1_X$. By restricting to the orbit closure $(\tau^n Z)_{n \in \mathbb{N}}$ of $Z$, without the loss of generality, we can assume $(\tau^n Z)_{n \in N}$ is dense in $X$. Applying Lemma \ref{lemma:main-integral-combination}, the null component of the sequence $\int_G f(\pi(g^n)) \, d \rho_V(g)$ is null along any EGEN sequence. 
    
    A convergent countable sum of nilsequences is a nilsequence. Likewise, a convergent countable sum of null sequences along any EGEN is null sequence along EGEN sequence. Therefore 
    \[
        \varphi(n) = \int_G f(\pi(g^n)) \, d \rho(g) = \sum_{V \in \mathcal{V}} \int_G f(\pi(g^n)) \, d \rho_V(g)
    \]
    has a nil+null decomposition, and its null component is null along any EGEN sequence. This finishes our proof of Proposition \ref{proposition:reduction-to-integration-on-nilgroup2}. \qedhere

\section{Null along good sequences}
    \label{section:strong-nil-null-decomposition}
    
    In this section, we prove the null component of a polynomial correlation is null along any polynomial-good sequence. By the same proof, the null component of a linear correlation is null along any linear-good sequence.  

\subsection{In a nilsystem}
    \label{section:decomposition-in-ergodic-nilsystems}

    \begin{proposition}
        \label{proposition:decomposition-in-ergodic-nilsystems}
        The null component of a polynomial correlation in a nilsystem is null along any EGEN sequence.
    \end{proposition}
    
    \begin{proof}
        Let $(X = G/\Gamma, \mu_X, \tau)$ be a nilsystem, $f_j \in L^{\infty}(\mu_X)$ and $s_j \in \mathbb{Z}[n]$. Let
        \[
            a(n) = \int_X T^{s_0(n)} f_0 \dots T^{s_k(n)} f_k \, d \mu
        \]
        
        By approximation we can assume $f_j$ is a continuous function on $X$ for all $j$. Then
        \begin{multline*}
            a(n) = \int_X f_0(\tau^{s_0(n) x} \dots f_k(\tau^{s_k(n) x} \, d \mu_X(x) =\\
            \int_X f_0 \otimes \dots \otimes f_k ((\tau^{s_0(n)}, \dots, \tau^{s_k(n)}) (x, x, \dots, x)) \, d \mu_X(x)
        \end{multline*}
        
        The function $F = f_0 \otimes \dots \otimes f_k$ is continuous on $X^{k+1}$. On the other hand, the sequence $(\tau^{s_0(n)}, \dots, \tau^{s_k(n)}) = g_0^{s_0(n)} \dots g_k^{s_k(n)}$ is a polynomial sequence on $G^{k+1}$ where $g_j = (1_G, \dots, 1_G, \tau_j, 1_G, \dots, 1_G)$. Hence $f_0 \otimes \dots \otimes f_k ((\tau^{s_0(n)}, \dots, \tau^{s_k(n)}) (x, x, \dots, x))$ is a polynomial nilsequence for all $x \in X$. By \cite{leib05-2}, a polynomial nilsequence is also a nilsequence (of higher degree of nilpotency). Therefore, $(a(n))$ is an integral of nilsequences. By Proposition \ref{proposition:integral-of-nilsequences}, the null component of $(a(n))$ is null along any EGEN sequence. Our proof finishes.
    \end{proof}
  
\subsection{In an ergodic system}

    \begin{lemma}   
    \label{lemma:nullity-along-good-sequence-in-ergodic-system-1}
        Let $(a(n))$ be a polynomial correlation in an ergodic system and $(r_n)$ be a polynomial-GPN sequence. Then there exists an $m \in \mathbb{N}$ such that if $(\tilde{a}(n))$ is the projection of $(a(n))$ on to $m$-step nilfactor $\mathcal{Z}_m$, then $a - \tilde{a}$ is a null sequence along $(n)$ and $(r_n)$. 
    \end{lemma}
    
    \begin{proof}
        Since both $(n)$ and $(r_n)$ are polynomial-GPN sequences, there exists $m_1$ and $m_2$ such that $\mathcal{Z}_{m_1}$ is characteristic for $(s_0(n), \dots, s_k(n))$ and $\mathcal{Z}_{m_2}$ is characteristic for $(s_0(r_n), \ldots, s_k(r_n))$. Let $m = \max \{m_1, m_2\} + 1$ and $\tilde{a}$ be the projection of $a$ on to $\mathcal{Z}_m$. Then with the same proof as in \cite[Corollary 4.5]{Bergelson_Host_Kra05}, we obtain the conclusion.
    \end{proof}

    \begin{proposition}
        \label{proposition:decomposition-in-ergodic-systems}
        The null component of a polynomial correlation in an ergodic system is null along any polynomial-good sequence.
    \end{proposition}
    
    \begin{proof}
        Let $(a(n))$ be a polynomial correlation in an ergodic system and $(r_n)$ be a polynomial-good sequence. The goal is to show the null component of $(a(n))$ is null along $(r_n)$. Let $ (\tilde{a}(n))$ as in Lemma \ref{lemma:nullity-along-good-sequence-in-ergodic-system-1}. Then by this lemma, $a - \tilde{a}$ is null along $(n)$ and $(r_n)$. 
        
        On the other hand, $\tilde{a}$ is a polynomial correlation arising from a nilfactor $Y$. $Y$ is an inverse limit of nilsystems, say $Y = \varprojlim Y_l$. Let $\tilde{a}_l$ be the projection of $\tilde{a}$ on to $Y_l$ for each $l \in \mathbb{N}$. By Proposition \ref{proposition:decomposition-in-ergodic-nilsystems}, $\tilde{a}_l$ can be written as $\tilde{a}_l = \psi_l + \epsilon_l$ where $\psi_l$ is a nilsequence and $\epsilon_l$ is null along any EGEN sequence, in particular along $(n)$ and $(r_n)$. Note that $\tilde{a}_l$ converges to $\tilde{a}$ uniformly as $l \to \infty$ since $Y = \varprojlim Y_l$. Hence it is easy to see $\psi_l$ converges to a nilsequence $\psi$ uniformly (see \cite[Section 7.4]{Bergelson_Host_Kra05}). Likewise, $\epsilon_l$ converges uniformly to a null sequence $\epsilon$ which is also null along $(r_n)$. Now $\tilde{a}$ can be written as $\tilde{a} = \psi + \epsilon$. Hence it has nil+null decomposition, and its null component is null along $(r_n)$.  
          
        In summary, we just show that both $a - \tilde{a}$ and the null component of $\tilde{a}$ is null along $(r_n)$. Therefore the null component of $a = \tilde{a} + (a - \tilde{a})$ is null along $(r_n)$. Our proof finishes.
    \end{proof}
    
\subsection{In a general measure preserving system}
    \begin{proposition}
        \label{proposition:nullity-along-polynomial-good-sequence-general-system}
        The null component of a polynomial correlation is null along any polynomial-good sequence.
    \end{proposition}
    
    \begin{proof}
        Let $\mu = \int_{\Omega} \mu_{\omega} \, d P(\omega)$ be the ergodic decomposition of $\mu$ with respect to $T$. Then $a(n) = \int_{\Omega} a_{\omega}(n) \, d P(\omega)$ where
        \[
            a_{\omega}(n) = \int_X T^{s_0(n)} f_0 \dots T^{s_k(n)} f_k \, d \mu_{\omega}
        \]
        
        For almost every $\omega \in \Omega$, the system $(X, \mu_{\omega}, T)$ is ergodic. Hence by Proposition \ref{proposition:decomposition-in-ergodic-systems}, the null component of $a_{\omega}$ is null along any polynomial-good sequence. To be precise, $a_{\omega} = \psi_{\omega} + \epsilon_{\omega}$ where $\psi_{\omega}$ is a nilsequence and $\epsilon_{\omega}$ is a null sequence along every polynomial-good sequence. Thus $a = \int_{\Omega} \psi_{\omega} \, d P + \int_{\Omega} \epsilon_{\omega} \, d P$.
        
        On one hand, the sequence $\int_{\Omega} \psi_{\omega} \, d P$ is an integral of nilsequences. Hence its null component is null along any polynomial-good sequence by Proposition \ref{proposition:integral-of-nilsequences}. On the other hand $\int_{\Omega} \epsilon_{\omega} \, d P$ is obviously null along any polynomial-good sequence (integral of null sequences is still a null sequence). Therefore, the null component of $(a(n))$ is null along any polynomial-good sequence. Our proof finishes.
    \end{proof}
    
    \section{Polynomials of primes are polynomial-good sequences}
    \label{section:polynomial-of-primes-is-good}
    
In this section, we show that polynomiasl of primes are a polynomial-good sequences. The proof has two parts. One is to show such sequences are polynomial-GPN. The other one is to show they are EGEN. Some notations are needed before going into the details.

The \emph{modified von Mangoldt function} is defined to be $\Lambda'(n) = \begin{cases} \log n \; \mbox{ ($n \in \mathbb{P}$)} \\ 0 \;\; (n \in \mathbb{N} \setminus \mathbb{P}) \end{cases}$

The \emph{Euler totient function} $\phi(n)$ is the number of positive integer not greater than $n$ and relatively prime to $n$.

For $r < M \in \mathbb{N}$, define
\[
    \Lambda'_{M, r} (n) = \frac{\phi(M)}{M} \Lambda'(Mn + r)    
\]

For $\omega \in \mathbb{N}$, define $W = \prod_{p \in \mathbb{P}, p < \omega} p$.
    
The symbol $o_{\omega \rightarrow \infty}(1)$(or $o_{N \rightarrow \infty}(1)$) represents a function of $\omega$ that approaches zero as $\omega \rightarrow \infty$($N \rightarrow \infty$ respectively). Furthermore, $o_{\omega, N \rightarrow \infty}(1)$ is a function of $\omega$ and $N$ such that for a fixed $\omega$, the function approaches zero as $N \rightarrow \infty$.

For two sequence $a, b$ by writing $a(N) \sim b(N)$ we mean  $\lim_{N \rightarrow \infty} a(N)/b(N) = 1$. 

\subsection{Polynomial of primes are polynomial-GPN}
Let $Q(n), s_j(n) \in \mathbb{Z}[n]$ for $1 \leq j \leq k$. The results of Host-Kra \cite{hk05-2} and Leibman \cite{leib05-3} show that there exists some $m \in \mathbb{N}$ such that $m$-step nilfactor $\mathcal{Z}_m$ is characteristic for $(s_1(Q(n)), \dots, s_k(Q(n))$. Moreover $m$ only depends on the degrees of $s_j(Q(n))$. Now we prove with the same $m$, $\mathcal{Z}_m$ is characteristic for $(s_1(Q(p_n)), \dots, s_k(Q(p_n))$. Denote
\[
    a(n) = T^{s_1(Q(n))} f_0 \dots T^{s_k(Q(n))} f_k \in L^{\infty}(\mu)
\]

\[
    B_{W, r}(N) = \mathbb{E}_{n \in [N]} a(Wn +r)
\]

The key ingredient in our proof is a proposition from Frantzikinakis-Host-Kra \cite{Frantzikinakis_Host_Kra13} that compares ergodic averages along primes to the averages along integers.

\begin{lemma}[Frantzikinakis-Host-Kra {\cite[Proposition 3.6]{Frantzikinakis_Host_Kra13}}]
    \label{lemma:1}
    \[
        \max_{r < W, (r,W) = 1} \left \Vert \mathbb{E}_{n \in [N]} \Lambda'_{W, r}(n) a(Wn+r) - B_{W, r}(N) \right \Vert_{L^2(\mu)} = o_{N \to \infty, \omega}(1) + o_{\omega \to \infty}(1)
    \]
\end{lemma}

We are ready for the main result of this section.

\begin{proposition}
    \label{proposition:poly-prime-GPN}
    For any $Q \in \mathbb{Z}[n]$ non-constant, the sequence $(Q(p_n))$ is polynomial-GPN.
\end{proposition}

\begin{proof}
    With the notation as before, assume $\mathbb{E}(f_j|\mathcal{Z}_m) = 0$ for some $j$. We need to prove 
    \[
        \lim_{N \to \infty} \mathbb{E}_{[N]} a(p_n) = 0 \;\;\; \mbox{ in } L^2(\mu). 
    \]
    
    From Lemma \ref{lemma:1}, taking average along $r < W, (r,W) = 1$, we get
    \begin{multline}
    \label{equation:poly-prime-1}
        \left \Vert \mathbb{E}_{(r,W) =1} \mathbb{E}_{n \in [N]} \Lambda'_{\omega, r}(n) a(Wn+r) - \mathbb{E}_{(r,W) =1} B_{\omega, r}(N) \right \Vert_{L^2(\mu)} = \\ o_{\omega, N \to \infty}(1) + o_{\omega \to \infty}(1)
    \end{multline}
    
    Note that 
    \[
        \mathbb{E}_{(r,W) = 1} \mathbb{E}_{n \in [N]} \Lambda'_{\omega, r}(n) a(Wn+r) = \mathbb{E}_{n \in [WN]} \Lambda'(n) a(n)
    \]
    
    Hence (\ref{equation:poly-prime-1}) becomes
    \begin{equation}
    \label{equation:poly-prime-1-2}
        \left \Vert \mathbb{E}_{n \in [WN]} \Lambda'(n) a(n) - \mathbb{E}_{(r,W) =1} B_{\omega, r}(N) \right \Vert_{L^2(\mu)} = o_{\omega, N \to \infty}(1) + o_{\omega \to \infty}(1)
    \end{equation}
    
    And since $\mathcal{Z}_m$ is characteristic for $(s_1(Q(Wn + r)), \dots, s_k(Q(Wn+r)))$ for any $W$ and $r$ (remember that $m$ only depends on the degrees of the polynomials), we have $\lim_{N \to \infty} B_{\omega, r}(N) = 0$ in $L^2(\mu)$. Hence
    \[
        \lim_{N \to \infty} \mathbb{E}_{(r,W) =1} B_{\omega, r}(N) = 0 \;\; \mbox{ in } L^2(\mu).
    \]
    
    On the other hand, by results of Wooley-Ziegler \cite{wz12}, Frantzikinakis-Host-Kra \cite{Frantzikinakis_Host_Kra13}, the limit
    \[
        \lim_{N \to \infty} \mathbb{E}_{[N]} \Lambda'(n) a(n)
    \]
    exists in $L^2(\mu)$ (equal to $ \lim_{N \to \infty} \mathbb{E}_{[N]} a(p_n)$). Call that limit $F \in L^2(\mu)$. 
    
    Then for any $W \in \mathbb{N}$, 
    \[
        F = \lim_{N \to \infty} \mathbb{E}_{[WN]} \Lambda'(n) a(n).
    \]
    
    Taking limit as $N \to \infty$ in (\ref{equation:poly-prime-1-2}), we get
    \begin{equation}
        \label{equation:poly-prime-2}
        \left \Vert F - 0 \right \Vert_{L^2(\mu)} = o_{\omega \to \infty}(1).
    \end{equation}
    
    The left hand side of (\ref{equation:poly-prime-2}) no longer depends on $\omega$. Let $\omega \to \infty$, we get $F = 0$ in $L^2(\mu)$. This finishes our proof. 
\end{proof}
\subsection{Polynomial of primes is EGEN}

We need the following proposition by Green and Tao.

\begin{proposition}[Green-Tao {\cite[Theorem 7.1]{gt12}}]
    \label{proposition:green-tao1}
    For sufficiently large $\omega \in \mathbb{N}$, define $W = W(\omega) = \prod_{p \in \mathbb{P}, p < \omega} p$. Suppose $(X = G/\Gamma, g)$ is a nilsystem, $x \in X$ and $F \in C(X)$. Then
    \[
        \max_{b < W, (b,W) = 1} \left| \lim_{N \rightarrow \infty} \mathbb{E}_{n \in [N]} \left( \Lambda'_{W,b}(n) - 1 \right) F(g^{Q(n)} x ) \right| = o_{\omega \rightarrow \infty}(1)
    \]
\end{proposition}

\begin{remark*}
    Green-Tao's version of Proposition \ref{proposition:green-tao1} is slightly different to the one we introduce here. In Green-Tao's, in the place of $F(g^{Q(n)}) x$ is a Lipschitz nilsequence that arises from a connected and simply-connected nilpotent group $G$. However, it is immediate to see Green-Tao's version implies ours. First, as discussed in Section \ref{subsec:nilseq}, any basic nilsequence can be seen as arising from a nilmanifold whose Lie group is connected and simply-connected. Second, every polynomial nilsequence is a nilsequence (See \cite[Thm. B* Proof]{leib05-2}). Lastly, we can take $F$ to be any continuous function since Lipschitz functions are dense in $C(X)$.

\end{remark*}

We have a corollary.

\begin{corollary} 
    \label{corollary:green-tao-d}
    Fix $d \in \mathbb{N}$. For sufficiently large $\omega \in \mathbb{N}$, define $W = W(\omega) = \prod_{p \in \mathbb{P}, p < \omega} p$. Suppose $(X = G/\Gamma, g)$ is a nilsystem, $x \in X$ and $F \in C(X)$. Then 
        \[
            \max_{b < dW, (b,dW) = 1} \left| \lim_{N \rightarrow \infty} \mathbb{E}_{n \in [N]} \left( \Lambda'_{dW,b}(n) - 1 \right) F(g^{Q(n)} x) \right| = o_{\omega \rightarrow \infty}(1)
        \]
\end{corollary}
\begin{proof}
    By \cite{leib05-2}, sequence $(F^{Q(n)} x)_{n \in \mathbb{Z}}$ is a basic nilsequence. Therefore, according to Leibman \cite[Lemma 2.4]{leib10}, there exists a basic nilsequence $(F'(g'^{n} x'))_{n \in \mathbb{N}}$ such that $F'(g'^{dn + r} x') = F(g^{Q(n)} x)$ for $r = 0$ and $F'(g'^(dn+r) x') = 0$ for $1 \leq r \leq d-1$ for all $n \in \mathbb{N}$. Note that for sufficient large $\omega$, all prime divisors of $d$ divide $W$. Therefore for $b \in \mathbb{N}$ we have $(b ,W) = 1$ if and only if $(b, dW) = 1$. It follows that $\phi(dW) = d \phi(W)$. Thus $\phi(dW)/(dW) = \phi(W)/W$. By Proposition \ref{proposition:green-tao1}, we have
    \begin{equation}
    \label{eq:1}
        \max_{b < dW, (b, dW) =1} \left| \lim_{N \rightarrow \infty} \mathbb{E}_{n \in [N]} \left( \frac{\phi(dW)}{dW} \Lambda'(Wn + b) - 1 \right)  F'(g'^n x') \right| = o_{\omega \rightarrow \infty}(1).
    \end{equation}
    
    Note that the sequence $(F'(g'^n x'))_{n \in \mathbb{N}}$ is supported on the set $\{n \in \mathbb{N}: n = dm \mbox{ for some } m \in \mathbb{N}\}$. Replacing $n$ by $dm$, the left hand side of Equation \ref{eq:1} is now equal to
    \[
        \frac{1}{d} \max_{b < dW, (b, dW) =1} \left| \lim_{N \rightarrow \infty} \mathbb{E}_{m \in [N]} \left( \frac{\phi(dW)}{dW} \Lambda'(Wdm + b) - 1 \right)  F'(g'^{dm} x') \right|.
    \]
    Our proof is finished by noting that $F'(g'^{dm} x') = F(g^{Q(m)} x)$ for all $m \in \mathbb{N}$.
\end{proof}

We need two more lemmas before turning to the main theorem.

\begin{lemma}
    \label{lem1}
        Let $a(n)$ be a bounded sequence. Define $Q_{r,d}(N) := \{ 1 \leq q \leq N: qd+r \in \mathbb{P}\}$. Then 
        \[
            \mathbb{E}_{q \in Q_{r,d}(N)} a(q) - \mathbb{E}_{n \in [N]} \Lambda'_{d,r}(n) a(n) = o_{N \rightarrow \infty}(1). 
        \]
    \end{lemma}
    \begin{proof}
        Let $\pi_{r,d}(N)$ be the cardinality of $Q_{r,d}(N)$. By Dirichlet's Theorem about primes on arithmetic progession, the set of primes that are congruent to $r \pmod d$ has density $\frac{1}{\phi(d)}$ in the set of primes. Therefore,
         \[
            \pi_{r,d}(N) \sim \frac{\pi(dN + r)}{\phi(d)} \sim \frac{dN + r}{\log(dN + r) \phi(d)} \sim \frac{dN}{(\log N) \phi(d)}.
        \]
        Note that $\Lambda'(dn+r) = 0$ if $dn + r \not \in \mathbb{P}$. Therefore
        \begin{multline*}
            u(N) := \left| \mathbb{E}_{q \in Q_{r,d}(N)} a(q) - \mathbb{E}_{n \in [N]} \frac{\phi(d)}{d} \Lambda'(dn + r) a(n) \right|\\
            = \left| \frac{1}{N}  \sum_{q \in Q_{r,d}(N)} a(q) \left(  \frac{N}{\pi_{r,d}(N)} - \frac{\phi(d)}{d} \Lambda'(dq+r) \right) \right|\\
            = \left| \frac{1}{N} \sum_{q \in Q_{r,d}(N)} a(q) \left( \frac{\phi(d)}{d} \log N - \frac{\phi(d)}{d} \log (dq + r) \right) \right| + o_{N \rightarrow \infty}(1)
        \end{multline*}
        
        Note that $\log(dq + r) = \log(q) + \log(d + r/q)$. And we have
        \begin{multline*}
            \left| \frac{1}{N} \sum_{q \in Q_{r,d}(N)} a(q) \log \left( d + \frac{r}{q} \right) \right| \leq \left \Vert a \right \Vert_{\infty} \frac{1}{N} \sum_{q \in Q_{r,d}(N)} \log \left( d + \frac{r}{q} \right) \\
            = \left \Vert a \right \Vert_{\infty} \frac{\pi_{r,d}(N) \log (d+ r/q)}{N} \\ =\left \Vert a \right \Vert_{\infty} \frac{(Nd + r) \log(d + r/q)}{\phi(d) \log(Nd+r) N} + o_{N \to \infty}(1) = o_{N \to \infty}(1).
        \end{multline*}
        
        Therefore,
        \[
            u(N) = \left| \frac{1}{N} \sum_{q \in Q_{r,d}(N)} a(q) \left( \frac{\phi(d)}{d} \log N - \frac{\phi(d)}{d} \log (q) \right) \right| + o_{N \rightarrow \infty}(1)
        \]
        
        Since $\log q < \log N$ for all $q \in Q_{r,d}(N)$, we have
        \begin{multline*}
            u(N) \leq \left \Vert a \right \Vert_{\infty} \left( \frac{1}{N} \sum_{q \in Q_{r,d}(N)} \frac{\phi(d)}{d} \log N  - \frac{1}{N} \sum_{q \in Q_{r,d}(N)} \frac{\phi(d)}{d} \log(q) \right) + o_{N \rightarrow \infty}(1) \\
            = \left \Vert a \right \Vert_{\infty} \left( \frac{1}{N} \sum_{q \in Q_{r,d}(N)} \frac{\phi(d)}{d} \log N  - \frac{1}{N} \sum_{q \in Q_{r,d}(N)} \frac{\phi(d)}{d} \log(qd + r) \right) + o_{N \rightarrow \infty}(1)\\
            = \left \Vert a \right \Vert_{\infty} \left( \frac{1}{N} \sum_{q \in Q_{r,d}(N)} \frac{\phi(d)}{d} \log N  - \frac{1}{N} \sum_{q \in Q_{r,d}(N)} \frac{\phi(d)}{d} \Lambda'(qd+r) \right) + o_{N \rightarrow \infty}(1)
        \end{multline*}
        
        By Dirichlet's Theorem, 
        \[
            \frac{1}{N} \sum_{q \in Q_{r,d}(N)} \frac{\phi(d)}{d} \log N = \frac{\phi(d) \log N}{dN} \pi_{r,d}(N) = 1 + o_{N \to \infty}(1)
        \]
        On the other hand, by Siegel-Walfisz's Theorem \cite{Walfisz1936} on average of the von Malgoldt functions on arithmetic progressions, 
        \[
            \frac{1}{N} \sum_{q \in Q_{r,d}(N)} \frac{\phi(d)}{d} \Lambda'(qd+r) = 1 + o_{N \to \infty}(1)
        \]
        So $u(N) \leq o_{N \to \infty}(1)$. Our proof finishes. 
\end{proof}

\begin{proposition}
    \label{proposition:prime-poly-q_n}
    Fix $d \in \mathbb{N}$ and $r \in \{0,1, \ldots, d-1\}$. For $n \in \mathbb{N}$, define $q_n$ to be the $n^{th}$ integer such that $d q_n + r$ is a prime. Let $(X = G/\Gamma, \mu_X, \tau)$ be a totally ergodic nilsystem. Then for any $Q(n) \in \mathbb{Z}[n]$ non-constant and $x \in X$, the sequence $(\tau^{Q(q_n)} x)_{n \in \mathbb{N}}$ is equidistributed on $X$.
\end{proposition}

\begin{proof}
    Let $f \in C(X)$. Replacing $f$ by $f - \int f $, we can assume $\int f = 0$. For any $x \in X$, we want to show
    \[ 
        \lim_{N \rightarrow \infty} \mathbb{E}_{n \in [N]} f(\tau^{Q(q_n)} x) = 0.
    \]

    For sufficiently large $\omega \in \mathbb{N}$, let $B_{\omega} = \{0 \leq b < dW: (b, dW) = 1, b \equiv r \pmod d \}$. By Corollary \ref{corollary:green-tao-d}, for $b \in B$,
    \[
        \lim_{N \rightarrow \infty} \mathbb{E}_{n \in [N]} \left( \Lambda'_{dW,b}(n) -1 \right) f\left( \tau^{Q(Wn + \frac{b-r}{d})} x \right) = o_{\omega \rightarrow \infty}(1).
    \]
    Note that the term $o_{\omega \rightarrow \infty}(1)$ does not depend on $b$. Therefore
    
    \begin{multline}
    \label{equation:poly-prime-3}
        \lim_{N \rightarrow \infty} \mathbb{E}_{n \in [N]} \Lambda'_{dW, b}(n) f\left(\tau^{Q(Wn + \frac{b-r}{d})} x \right) = \\
        =\lim_{N \rightarrow \infty} \mathbb{E}_{n \in [N]} f\left(\tau^{Q(Wn + \frac{b-r}{d})} x \right) + o_{\omega \rightarrow \infty}.
    \end{multline}
    Since $(X, \mu_X, g)$ is totally ergodic, for any $b \in B$, the sequence $(\tau^{Q(Wn + \frac{b-r}{d})})$ is equidistributed on $X$ (see Section ??). That means
    \[
        \lim_{N \rightarrow \infty} \mathbb{E}_{n \in [N]} f\left( \tau^{Q(Wn + \frac{b-r}{d})} x \right) = 0.
    \]
    Hence, (\ref{equation:poly-prime-3}) implies 
    \begin{equation}
        \label{equa1}
        \lim_{N \rightarrow \infty} \mathbb{E}_{n \in [N]} \Lambda'_{dW,b}(n) f \left( \tau^{Q(Wn + \frac{b-r}{d})} x \right) = o_{\omega \rightarrow \infty}(1).
    \end{equation}
    
    Let $B'_{\omega} = \{ 0 \leq b < dW: b \equiv r \pmod d\}$. In (\ref{equa1}), summing over $b \in B_{\omega}'$ and noting that $\Lambda'(dWn + b) = 0$ if $b \in B_{\omega}' \setminus B_{\omega}$, we get
    \begin{equation}
        \sum_{b \in B_{\omega}'} \lim_{N \rightarrow \infty} \mathbb{E}_{n \in [N]} \Lambda'_{dW, b}(n) f\left( \tau^{Q(Wn + \frac{b-r}{d})} x \right)
        = |B_{\omega}| o_{\omega \rightarrow \infty}(1).
    \end{equation}
    
    Now dividing both sides by $W$ (which is the cardinality of $B_{\omega}'$), we get
    \begin{equation}
    \label{equation:poly-prime-4}
        \lim_{N \rightarrow \infty} \mathbb{E}_{n \in [N]} \mathbb{E}_{b \in B_{\omega}'} \Lambda'_{dW,b}(n) f\left(\tau^{Q(Wn + \frac{b-r}{d})} x \right) = \frac{|B_{\omega}|}{W} o_{\omega \rightarrow \infty}(1).
    \end{equation}
    
    The left hand side now is equal to 
    \[
        \lim_{N \rightarrow \infty} \mathbb{E}_{n \in [WN]} \frac{\phi(dW)}{dW} \Lambda'(dn + r) f\left( \tau^{Q(n)} x \right)
    \]
    
    Multiplying both sides of (\ref{equation:poly-prime-4}) by $\phi(d) W/\phi(dW)$, we get
    \begin{equation}
    \label{equation:poly-prime-5}
        \lim_{N \rightarrow \infty} \mathbb{E}_{n \in [WN]} \Lambda'_{d,r}(n) f\left(\tau^{Q(n)} x \right) =  \frac{\phi(d) |B_{\omega}|}{\phi(dW)} o_{\omega \rightarrow \infty}(1) = o_{\omega \to \infty}(1).
    \end{equation}
    
    Note that
    \begin{equation*}      
        \mathbb{E}_{n \in [N]} \Lambda'_{d,r}(n) f(\tau^{Q(n)} x) = \mathbb{E}_{n \in [W \lfloor N/W \rfloor]} \Lambda'_{d,r}(n) f(\tau^{Q(n)} x) + o_{\omega, N \rightarrow \infty}(1).
    \end{equation*}
    
    Therefore (\ref{equation:poly-prime-5}) implies
    \[
        \lim_{N \rightarrow \infty} \mathbb{E}_{n \in [N]} \Lambda'_{d,r}(n) f\left(\tau^{Q(n)} x \right) = o_{\omega \rightarrow \infty}(1).
    \]
    
    The left hand side no longer depends on $\omega$. By letting $\omega$ approaches infinity, we get
    \[
        \lim_{N \rightarrow \infty} \mathbb{E}_{n \in [N]} \Lambda'_{d,r}(n) f\left(\tau^{Q(n)} x \right) = 0.
    \]
    
     Applying Lemma \ref{lem1} to $a(n) = f(\tau^{Q(n)} x)$, our proposition is proved.
\end{proof}

We are ready to prove the main result of this section.

\begin{proposition}
    \label{proposition:poly-prime-EGEN}
    For $Q \in \mathbb{Z}[n]$ non-constant, the sequence $(Q(p_n))$ is EGEN.
\end{proposition}
\begin{proof}
    Let $(X = G/\Gamma, \mu_X, \tau)$ be an ergodic nilsystem. Suppose $X$ has $d$ connected components and $X_0$ is the component containing $1_X$. For $(r,d) = 1$, we need to show the sequence $(\tau^{Q(p_{r \pmod d, n})} 1_X)_{n \in \mathbb{N}}$ is equidistributed on $\tau^{Q(r)} X_0$.
    
    Let $q_n = (p_{r \pmod d, n} - r)/d$. Applying Proposition \ref{proposition:prime-poly-q_n} to totally ergodic nilsystem $(\tau^{Q(r)} X_0, \tau^{Q(r)} \mu_{X_0}, \tau^d)$, polynomial $P(n) = (Q(dn+r)-Q(r))/d \in \mathbb{Z}[n]$, we get sequence $(\tau^{P(q_n)} \tau^{Q(r)} 1_X)$ is equidistributed on $\tau^{Q(r)} X_0$. Then our proposition follows.
\end{proof}

\subsection{Proof of Proposition \ref{proposition:poly-prime-good}}

    By Proposition \ref{proposition:poly-prime-GPN}, $(Q(p_n))$ is polynomial-GPN. By Proposition \ref{proposition:poly-prime-EGEN}, the same sequence is EGEN. Hence by definition, $(Q(p_n))$ is polynomial-good.
    
\subsection{Proof of Theorem \ref{theorem:1}}

By Proposition \ref{proposition:nullity-along-polynomial-good-sequence-general-system}, the null component of a polynomial correlation is null along any polynomial-good sequence. The sequence $(Q(n))$ is polynomial-good as pointed out in Section \ref{subsection:good-sequences}. On the other hand, the sequence $(Q(p_n))$ is polynomial-good by Proposition \ref{proposition:poly-prime-good}. Hence the first part of Theorem \ref{theorem:1} is proved. Same argument applies to linear correlations and linear-good sequence $(\lfloor n^c \rfloor)$. That proves the second part of the theorem. 

\subsection{Proof of Corollary \ref{corolarry:average-along-prime-formula}}

Let $(X = G/\Gamma, \mu, \tau)$ be an ergodic nilsystem with $d$ connected component and the set-up as in Corollary \ref{corolarry:average-along-prime-formula}. Since $(Q(p_n))$ is EGEN, for each $s < d$ with $(s, d) =1 $, the sequence $(\tau^{Q(p_{s \pmod d,n})} x)$ is equidistributed on $X_{Q(s) + k}$. Hence
\begin{equation}
    \label{equation:average-along-prime-formula-1}
    \mathbb{E}_{p \in \mathbb{P}, p \equiv s \pmod d} f(\tau^{Q(p)} x) = \int_{X_{Q(s)+k}} f \, d \mu_{X_{Q(s) + k}}
\end{equation}
For each $s < d$ with $(s,d) = 1$, the set $\{p \in \mathbb{P}: p \equiv s \pmod d\}$ has density $1 /\phi(d)$. Taking average for all $s$, we have Corollary \ref{corolarry:average-along-prime-formula}.

\section{Proof of Proposition \ref{proposition:3^n}}
\label{section:not-null-along-3^n}

In this section, we prove Proposition \ref{proposition:3^n}. Some preliminary facts are needed before going into the proof. By Herglotz's Theorem, for $f \in L^2(\mu)$, there exists a complex measure $\sigma$ on circle $\mathbb{T}$ such that $a(n) := \int_X f \cdot T^n \bar{f} \, d \mu = \int_{\mathbb{T}} e^{2 \pi i n t} \, d \sigma(t) =: \hat{\sigma}(n)$ for all $n \in \mathbb{Z}$. The measure $\sigma$ is called the spectral measure of $f$. Some times we denote it by $\sigma_f$ to indicates the dependence on $f$. By decomposing $\sigma$ to discrete and continuous parts, we get $a(n) = \hat{\sigma}_{d}(n) + \hat{\sigma}_{c}(n)$. The sequence $(\hat{\sigma}_d(n))_{n \in \mathbb{N}}$ and $(\hat{\sigma}_c(n))_{n \in \mathbb{N}}$ are the nil and null components of $(a(n))_{n \in \mathbb{N}}$, respectively. We have a proposition.

\begin{proposition}
	\label{proposition:not-null-stronger}
    Let $(r_n)$ be a increasing sequence of integers such that there is a continuous measure $\sigma$ on $\mathbb{T}$ such that $(\hat{\sigma}(n))$ is not null along $(r_n)$. Then there exists a linear correlation whose null component is not null along $(r_n)$.
\end{proposition}

\begin{proof}
	Let $(r_n)$ be a increasing sequence of integers and $\sigma$ be a continuous measure satisfy Proposition \ref{proposition:not-null-stronger} hypothesis. Our proposition would have been proved if there existed a bounded function $f$ on a system $(X, \mu, T)$ whose spectral measure $\sigma_f$ is equal to $\sigma$. However, we are only guaranteed the existence of an unbounded $L^2$-function (let call it $g$) from a Gaussian system (see Section \ref{subsection:Gaussian-system}) that satisfy the required condition. To achieve our goal, a small modification is needed. 
    
    Since $L^{\infty}(\mu)$ is dense in $L^2(\mu)$, for a small $\epsilon_1$, there exists $f \in L^{\infty}(\mu)$ such that $\lVert f - g \rVert_{L^2} < \epsilon_1$. That implies:
    \begin{equation}
    	\label{equation:6-1}
    	\lvert \hat{\sigma}_f(n) - \hat{\sigma}_g(n) \rvert = \left \vert \int f T^n \bar{f} - \int g T^n \bar{g} \right \vert < 2 \epsilon_1 \lVert g \rVert_{L^2} + \epsilon_1^2 =: \epsilon_2.
    \end{equation}
    
    Let $\sigma_{gd}$ and $\sigma_{gc}$ denote discrete and continuous parts of $\sigma_g$ respectively. By Wiener's Lemma, since $\sigma_g $ is continuous:
    \begin{equation}
    	\label{equation:6-2}
    	\sigma_{gd}(\mathbb{T}) = \lim_{N \to \infty} \mathbb{E}_{n \in [N]} \left \vert \hat{\sigma}_g(n) \right \vert^2 = 0.
    \end{equation}
    
    And,
    \begin{equation}
    	\label{equation:6-3}
    	\sigma_{fd}(\mathbb{T}) = \lim_{N \to \infty} \mathbb{E}_{n \in [N]} \left \vert \hat{\sigma}_f(n) \right \vert^2
    \end{equation}

	From (\ref{equation:6-1}), (\ref{equation:6-2}) and (\ref{equation:6-3}), we get $\sigma_{gd}(\mathbb{T}) < \epsilon_2^2$. That implies $\lvert \hat{\sigma}_{fd}(n) \rvert < \epsilon_2^2$ for all $n \in \mathbb{N}$. Hence for all $n \in \mathbb{N}$:
    \begin{equation}
    	\label{equation:6-4}
    	\lvert \hat{\sigma}_{fc}(n) - \hat{\sigma}_g(n) \rvert = \lvert \hat{\sigma}_f(n) - \hat{\sigma}_{fd}(n) - \hat{\sigma}_g(n)| < \epsilon_2 + \epsilon_2^2.
    \end{equation}
    
    Since $(\hat{\sigma}_g(n))$ is not null along $(r_n)$, there is $\delta > 0$ such that
    \begin{equation}
    	\label{equation:6-5}
    	\limsup \limits_{N \to \infty} \mathbb{E}_{n \in [N]} \lvert \hat{\sigma}_g(r_n) \rvert = \delta.
    \end{equation}
    
    If we choose $\epsilon_1$ sufficiently small so that $\epsilon_2 + \epsilon_2^2 < \delta$, then from (\ref{equation:6-4}) and (\ref{equation:6-5}), we have
    \begin{equation}
    	\limsup \limits_{N \to \infty} \mathbb{E}_{n \in [N]} \lvert \hat{\sigma}_{fc}(r_n) \rvert > \delta - \epsilon_2 - \epsilon_2^2 > 0.
    \end{equation}
    
    That would imply $(\hat{\sigma}_{fc}(n))$ is not null along $(r_n)$. Therefore, the null component of $\int f T^n \bar{f}$ is not null along $(r_n)$. Our proof finishes. 
\end{proof}

\textbf{Proof of Proposition \ref{proposition:3^n}}. If $(r_n)$ is a rigid sequence, then there is a continuous measure $\sigma$ on $\mathbb{T}$ such that $\hat{\sigma}(r_n) \to 1$ as $n \to \infty$ (see Section \ref{subsection:backgroun-rigidity-sequences}) . Hence $(\hat{\sigma}(n))$ is not null along $(r_n)$, so $(r_n)$ satisfies Proposition \ref{proposition:not-null-stronger} hypothesis.
	
    \begin{remark*}
	In a recent paper, Badea and Grivaux \cite{Badea-Grivaux2018} show that there exists a continuous measure $\sigma$ on $\mathbb{T}$ such that $\liminf_{n,m \in \mathbb{N}} \hat{\sigma}(2^n 3^m) > 0$. Hence sequence $(2^n 3^m)$ when ordering in the increasing fashion also satisfies Proposition \ref{proposition:not-null-stronger} hypothesis. 
    \end{remark*}

\bibliographystyle{plain}
\bibliography{nil-cor}

\end{document}